\input ./style/arxiv-general.cfg
\documentclass[aap,MSNbibl,seceqn,dvips]{arximspdf}
\makeatletter
   \@ifpackageloaded{graphicx}{}{\usepackage{graphicx}}
\makeatother

\doi{10.1214/15-AAP1123}
\volume{26}
\issue{3}
\pubyear{2016}
\firstpage{1467}
\lastpage{1494}
\docsubty{FLA}

\makeatletter
\newcommand{\rrvert}{\vert}
\newcommand{\rrVert}{\Vert}
\newcommand{\llvert}{\vert}
\newcommand{\llVert}{\Vert}
\newcommand{\eqref}[1]{(\ref{#1})}
\newtheorem{theorem}{Theorem}[section]

\newtheorem{lemma}[theorem]{Lemma}
\newtheorem{corollary}[theorem]{Corollary}

\newproclaim{remark}[theorem]{Remark}
\newproclaim{example}[theorem]{Example}
\newproclaim{definition}[theorem]{Definition}
\newproclaim{conjecture}[theorem]{Conjecture}
\newproclaim{question}[theorem]{Question}
\newproclaim{assumption}[theorem]{Assumption}

\newcommand{\Z}{\mathbb{Z}}
\newcommand{\N}{\mathbb{N}}
\newcommand{\R}{\mathbb{R}}
\newcommand{\E}{\mathbb{E}}
\renewcommand{\P}{\mathbb{P}}
\newcommand{\F}{\mathcal{F}}

\makeatother

\begin{document}
\begin{frontmatter}

\title{Local asymptotics for controlled martingales}
\runtitle{Local asymptotics for controlled martingales}

\begin{aug}
\author[A]{\fnms{Scott N.}~\snm{Armstrong}\thanksref{m1}\ead[label=e1]{armstrong@ceremade.dauphine.fr}}
\and
\author[B]{\fnms{Ofer}~\snm{Zeitouni}\corref{}\thanksref{T1,m2,m3}\ead[label=e2]{ofer.zeitouni@weizmann.ac.il}}
\runauthor{S.~N. Armstrong and O. Zeitouni}
\affiliation{Universit\'e Paris-Dauphine\thanksmark{m1},
Weizmann Institute of Science\thanksmark{m2}\\ and New York University\thanksmark{m3}}
\thankstext{T1}{Supported in part by a grant of the Israel Science Foundation and by the Henri Taubman Professorial chair at the Weizmann Institute.}
\address[A]{Ceremade (UMR CNRS 7534)\\
Universit\'e Paris-Dauphine\\
Place du Mar\'{e}chal de Lattre de Tassigny\\
75775 Paris Cedex 16\\
France\\
\printead{e1}}
\address[B]{Faculty of Mathematics\\
Weizmann Institute of Science\\
Rehovot 76100\\
Israel\\
and\\
Courant Institute of\\
\quad Mathematical Sciences\\
New York University\\
251 Mercer Street\\
New York, New York 10012\\
USA\\
\printead{e2}}
\end{aug}

%
\received{\smonth{2} \syear{2015}}
%
\revised{\smonth{5} \syear{2015}}

%
\begin{abstract}
We consider controlled martingales with bounded steps
where the controller is allowed at each step to choose the distribution of
the next step, and where the goal is to hit a fixed ball at the origin
at time~$n$. We show that the algebraic rate of decay (as $n$ increases
to infinity)
of the value function in the discrete setup coincides with its
continuous counterpart, provided a reachability assumption is satisfied.
We also study in some detail the uniformly elliptic case and obtain explicit
bounds
on the rate of decay. This generalizes and improves upon several
recent studies of the one dimensional case, and is a discrete analogue
of a stochastic control problem recently investigated in Armstrong and
Trokhimtchouck [\textit{Calc. Var. Partial Differential Equations} \textbf{38}
(2010) 521--540].
\end{abstract}

%
\begin{keyword}[class=AMS]
\kwd{60G42}
\kwd{93E20}
\end{keyword}
\begin{keyword}
\kwd{Stochastic control}
\kwd{martingale}
\kwd{nonlinear parabolic equation}
\end{keyword}
\end{frontmatter}

\section{Introduction}\label{sec1}
Consider a family of
(possibly multidimensional) martingales $\{M_n\}_{n\geq0}$
in discrete time, with $M_0=0$, equipped with their natural filtration
$\F_n$. What
is the maximal probability that, at time $n$, the martingale is in a prescribed
set $A$? Similarly, what is the minimal probability?
We will be focused on the asymptotic analysis of these quantities as
$n\to\infty$.

This problem can be cast as a stochastic control problem \cite{BS,DY},
by noting
that $M_{n+1}=M_n+\Delta_{n+1}$ where the \textit{law}
of $\Delta_{n+1}$ is adapted to the filtration
$\F_n$, and is considered as a control;
the martingale condition then restricts the control to satisfy that
$\E[\Delta_{n+1}\vert\F_n]=0$.
Already in dimension $1$, the \textit{quantitative} aspects of this question,
which have
recently received attention from several authors (see \cite{A,AT,GPZ}),
lead to some nontrivial (and, to us, counterintuitive) observations,
which we now explain, in a somewhat more
restrictive setup than developed in the rest of the paper.

Fix $\lambda\in(0,1)$ and let
$\mathfrak{M}_{1,\lambda}$ denote\setcounter{footnote}{1}\footnote{More formally,
let $\Sigma=[-1,1]$ and let $\mathcal{P}$ denote the
collection of probability measures on
$\Sigma^\N$ equipped with its Borel $\sigma$-algebra.
For $\P\in\mathcal{P}$,
and finite, ordered $I\subset\N$, let $\P_I$ denote the marginal of
$\P$ on $\Sigma^I$, with $\P_n:=\P_{\{1,\ldots,n\}}$;
write $\P_n=\P_{n-1}\ltimes\P_n^{(n-1)}$ for the disintegration
of $\P_n$.
Then, $\mathfrak{M}_{1,\lambda}$ is the collection of
$\P\in\mathcal{P}$
satisfying the conditions
$\int x \,d\P_n^{(n-1)}(x)=0$,
$\int x^2 \,d\P_n^{(n-1)}(x)\in[\lambda,1]$,
$\P_{n-1}$ almost surely. The sequence $\{\Delta_i\}_{i\geq1}$
is then the
canonical process associated with $\P$.}
the
collection of laws of discrete time
martingales as above
satisfying
%
\begin{equation}
\label{eq-rev21} |\Delta_i|\leq1\qquad \mbox{a.s.}
\end{equation}
and
%
\begin{equation}
\label{eq-rev22} \lambda\leq\E \bigl[ \Delta_i^2 \vert
\F_{i-1} \bigr] \leq1 \qquad\mbox{a.s.}
\end{equation}
In this setup, we will be interested in the $n\to\infty$
asymptotics of
%
\begin{equation}
\label{eq-new2} q_n:=\sup_{\P\in\mathfrak{M}_{1,\lambda}} \P\bigl(|M_n|
\leq1\bigr).
\end{equation}

We now come to the counterintuitive
observations concerning $q_n$ alluded to above.
Introduce the quadratic variation process
$V_n=\sum_{i=1}^{n} \E[ \Delta_i^2
\vert\F_{i-1}]$, and note that
$\lambda n\leq V_n\leq n$. By the martingale central limit theorem
(see, e.g., \cite{HH}, Theorem~3.4), $M_n/\sqrt{V_n}$ converges in
distribution
as $n\to\infty$ to
a standard Gaussian random variable. In particular,
$M_n$ is spread out at scale $\sqrt{n}$, uniformly in
$\P\in\mathfrak{M}_{1,\lambda}$ and, therefore, one may naively
expect that a form of a local CLT could also hold, that is, that there
exists a constant $C$ such that
for any $\P\in\mathfrak{M}_{1,\lambda}$ and all $n$,
\[
\P\bigl[|M_n|\leq1\bigr]\leq\frac{C}{n^{1/2}}.
\]
This belief turns out to be false, as was shown in \cite{GPZ}: if
$\lambda<1$ then there exist constants $\alpha=\alpha(\lambda)<1/2$ and
$C=C(\lambda)>0$
so that, for every $n$ one may construct\footnote{In fact, the martingale constructed in
\cite{GPZ} is a time-inhomogeneous Markov chain with state space $\Z$,
and the
estimate in \eqref{eq-rev23} is then uniform in the starting
state as long as $|M_0|\leq\sqrt{n}$. Taking the sequence of times
$n_k=n_{k-1}+n_{k-1}^2$ with $n_0=1$, and using during
the time interval $[n_{k-1},n_k]$ the transition probabilities
from \cite{GPZ}
corresponding to $n=n_k-n_{k-1}$, one obtains a martingale $\{M_n\}$
satisfying
\eqref{eq-rev21} and
\eqref{eq-rev22},
and such that \eqref{eq-rev23} holds for all $n=n_k$ and
all $k\geq1$.}
a martingale with law
$\P\in\mathfrak{M}_{1,\lambda}$ so that
%
\begin{equation}
\label{eq-rev23} \P\bigl[|M_n|\leq1\bigr]\geq\frac{C}{n^{\alpha}}.
\end{equation}
Our goal is to provide more precise results on $q_n$ (and
its higher dimensional generalizations).
%

As noted earlier, the problem described in
\eqref{eq-new2} is a stochastic control problem; in particular,
a dynamic programming equation for $q_n$ (with
initial condition $M_0=x$)
can be written down; indeed, such a dynamic
programming equation [see \eqref{eq-rev1}]
will play an important role
in our analysis.
However, two aspects of the problem sets
it apart from much of the stochastic control literature.
First, we are interested here in precise asymptotic
results, and not so much on existence and regularity
results for the dynamic programming equation
or on the structure of the optimal control.
Second, in much of the \textit{quantitative}
aspects of
stochastic control theory where exact performance bounds are
available,
one is given a specific martingale and the control
appears multiplicatively or additively in the evolution
of the process; this is not the case here.
See, for example, \cite{BMR} for a (continuous time) problem
where jump processes are involved and the control
influences the size of the jumps.

Returning to the multidimensional case, there is an analogue
of the above stochastic control problem in the continuous time/space setup,
namely in the context of control of diffusion processes. Specifically,
the controlled process is
%
\begin{equation}
\label{eq-rev1a} dX_t^\sigma=\sigma_t
\,dW_t,\qquad X_0^\sigma=x ,
\end{equation}
where $\sigma_t$ is an adapted control taking values in a subset
${\mathcal M}$
of
the space of matrices,
$W_\cdot$ is a ($d$-dimensional) Brownian motion, and the payoff
is $\E[g(X_T)]$ for some fixed $T$ and continuous, bounded function $g$.
As explained in
\cite{Krylov}, Chapter~4, the value function
\[
u(x,t)=\sup_{\sigma\in{\mathcal A}} \E^{x,T-t} \bigl[g
\bigl(X^\sigma_T\bigr)\bigr],
\]
where ${\mathcal A}$ is the set of adapted controls as above, satisfies the
dynamic programming equation
%
\begin{equation}
\label{eq-rev2} \partial_t u -\frac{1}2 \sup
_{\sigma\in{\mathcal M} } \bigl(\operatorname{Tr} \bigl(\sigma\sigma^T
D^2 u\bigr)\bigr)=0 ,\qquad u(x,0)=g(x).
\end{equation}
Here, $D^2u$ denotes the Hessian of $u$ and $\operatorname{Tr}$ is the trace
operator. Equation \eqref{eq-rev2} is
a fully nonlinear, parabolic partial differential equation;
see \cite{CC} for background.
The analogue of computing the asymptotics of $q_n$ then is the
problem of computing the asymptotics of $u$ when $g$ is a nonnegative
function of compact support, as $T\to\infty$.
See \cite{mc1,KPV} for early work in this direction.
A rather complete description of the asymptotics was given recently
by \cite{AT}.
In particular, it is shown there
that if the
control $\sigma$ is
restricted to be uniformly elliptic
and bounded above, then there exists $\alpha> 0$ such that
\[
\sup_{\sigma\in{\mathcal A}} \P \bigl[ \bigl|X_t^\sigma\bigr|\leq1
\bigr] \sim t^{-\alpha}.
\]
Since \eqref{eq-rev2} is invariant under the change of scale
$(x,t)\mapsto(\sqrt{\theta}x, \theta t)$,
one expects that, in fact,
\[
\theta^\alpha u(\sqrt{\theta}x,\theta t)\longrightarrow \Phi(x,t)
\qquad\mbox{as } \theta\to\infty,
\]
uniformly on compact subsets of $\R^d\times(0,\infty)$, where
$\Phi$ satisfies the scaling relation
$\Phi(x,t)=\theta^\alpha\Phi(\sqrt{\theta} x, \theta t)$ for every
$\theta>0$.
Indeed, this
is precisely what is proved in \cite{AT}.
The exponent $\alpha$ is determined by the solution to a nonlinear
eigenvalue problem and typically we have $\alpha< d/2$.

Our goal in this paper is to provide a similar analysis of the discrete
time setup. At a heuristic level, one expects that scaling the discrete time
problem would lead to the continuous diffusion setup. Note, however,
that when rescaling, the initial conditions become singular, preventing
a direct application of the continuous time theory and
representing a significant technical challenge.
Our analysis therefore builds on \cite{AT} but requires significant
modifications. We present here two corollaries of our main result,
Theorem~\ref{t.alpha}. In what follows, $|\cdot|$ denotes the
Euclidean norm.

In the first corollary, we consider uniformly elliptic martingales: the
conditional variance of the projection of the jump
in any direction is bounded below uniformly. This is the natural
generalization of the one dimensional setup discussed in \cite{GPZ}.

\begin{corollary}[(Uniformly elliptic martingales)]
\label{cor-elliptic}
Fix $\lambda\in(0,1]$ and $R\geq\sqrt{2d}$ and let $\mathfrak
{M}_{d,\lambda,R}$ denote the
collection of laws of discrete time
martingales of the form
\[
M_n=\sum_{i=1}^n
\Delta_i\in\mathbb{R}^d
\]
satisfying
%
\begin{equation}
|\Delta_i|\leq R \qquad\mbox{a.s.}
\end{equation}
and
%
\begin{eqnarray}
\lambda&\leq&\inf_{v\in\mathbb{R}^d, | v| =1} \E \biggl[ \frac{1}2 (
\Delta_i \cdot v )^2\Big \vert\F_{i-1} \biggr]
\nonumber
\\[-8pt]
\\[-8pt]
\nonumber
 &\leq&
\sup_{v\in\mathbb{R}^d, | v|
=1} \E \biggl[ \frac{1}2(\Delta_i
\cdot v )^2 \Big\vert\F_{i-1} \biggr] \leq1\qquad \mbox{a.s.}
\end{eqnarray}
Then there exist constants $\alpha=\alpha(d,\lambda)
>0$ and $C=C(d,\lambda,R)\geq1$ such that, for all $n$ sufficiently large,
%
\begin{equation}
\label{eq-ellip} 
\frac{1}{Cn^\alpha}\leq \sup_{\P\in\mathfrak{M}_{d,\lambda,R}} \P\bigl [
|M_n|\leq\sqrt {d} R \bigr] \leq \frac{C}{n^\alpha}.
\end{equation}
\end{corollary}

Note that in case $d=1$, one recovers \eqref{eq-rev23}, in the strong
form of providing an asymptotic rate of decay.
We discuss at the end of the introduction quantitative properties of
the exponent $\alpha=\alpha(d,\lambda)$.

In our second corollary, we consider a nonuniformly
elliptic martingale, where
the control influences the direction of the jumps
of the martingale but
not the magnitude. This
answers a question communicated to us by Peres; after the work
on this paper was completed, we learned of an independent, different
proof of
the corollary, due to Lee, Peres and Smart \cite{LPS}.
We obtain the result as a
consequence of our general strategy of comparison to the
PDE satisfied in the continuous time setup, whereas \cite{LPS} first
reduces the question to a problem in two dimensions and then use a
mixture of probabilistic
and analytical arguments to analyze the latter.

\begin{corollary}
\label{cor-fixednorm}
Let $M_n$ be an $\mathbb{R}^d$-valued martingale adapted to a filtration
$\F_n$ with $X_0 = 0$ which satisfies, for some $\lambda\in(0,1]$,
\[
 \P \bigl[ |M_{n+1} - M_n| \leq1 \bigr] = 1
\]
and
\[
 \E \bigl[ |M_{n+1} - M_n|^2
\vert\F_n \bigr] = \lambda^2.
\]
Then there exists $C=C(\lambda)\ge1$ such that
\[
 \P\bigl [ |M_n| \leq1 \bigr] \leq C n^{-1/2}.
\]
\end{corollary}

Of course, by scaling, the constant $1$ appearing inside the
probabilities in
Corollary~\ref{cor-fixednorm} can be changed to any fixed constant.
Note that the exponent $1/2$ in Corollary~\ref{cor-fixednorm} is sharp
in every dimension, as exhibited by the local CLT for a simple random
walk in one of the coordinate directions.

We conclude this \hyperref[sec1]{Introduction} with some comments concerning
the exponent $\alpha$ in Corollary~\ref{cor-elliptic}: we will prove below
that
%
\begin{equation}
\label{e.alphabounds} \frac{d\lambda}{2} \leq\alpha(d,\lambda) \leq\frac{(d-1)\lambda
}{2} +
\frac{1}2
\end{equation}
and each of the two inequalities in \eqref{e.alphabounds} is an
equality if and only if $\lambda=1$. Notice in particular that this
implies that $\alpha(d,\lambda) < d/2$ if $\lambda< 1$, which means
that the quantity $\sup_{\P\in\mathfrak{M}_{d,\lambda,R}} \P [
|M_n|\leq\sqrt{2} \,dR  ]$ decays at a slower rate than for a
simple random walk. It was previously observed in \cite{GPZ} in the
discrete setup for $d=1$ that $\alpha<1/2$ if $\lambda<1$.
We generalize this to arbitrary dimension and
obtain the statement that $\alpha< d/2$ for general controlled,
uniformly elliptic martingales, provided that the set of controls has
at least two elements. Both the latter statement as well as the
bounds \eqref{e.alphabounds} were proved in \cite{AT}, (3.20), in the
continuum framework, and they apply in our discrete setup since, as we
will see, our exponent $\alpha$ is the same as the one corresponding to
the minimal Pucci operator from \cite{AT}.

It is also of interest to study the behavior of the exponent $\alpha
(d,\lambda)$ as $\lambda\to0$. Here, the estimate \eqref
{e.alphabounds} is not very sharp on either side, and it turns out
that, except for a possible sub-algebraic correction, $\alpha
(d,\lambda
) \sim\lambda^{1/4}$. Precisely, for each $\delta\in(0,1/4)$, there
exist constants $C(d,\delta)>1$ and $c(d,\delta)>0$ such that, for
every $\lambda\in(0,1]$,
%
\begin{equation}
\label{e.alphabounds2} c \lambda^{1/4+\delta} \leq\alpha(d,\lambda) \leq C \lambda
^{1/4-\delta}.
\end{equation}
In particular, $\alpha(d,\lambda) \to0$ as $\lambda\to0$ and
\[
\lim_{\lambda\to0}\limsup_{n\to\infty} \frac{|\log\sup_{\P\in
\mathfrak{M}_{d,\lambda,R}}\P [ |M_n|\leq\sqrt{d}R
]|}{\log
n} =
0 ,
\]
which was previously proved for $d=1$ in \cite{GPZ}. The interpretation
is that, for a controlled martingale, the quantity $\P [
|M_n|\leq
\sqrt{d} R ]$ may decay at an arbitrarily slow (algebraic) rate in
$n$ provided that the set of controls is sufficiently rich. The
bounds \eqref{e.alphabounds2} are new and follow from test function
calculations in Section~\ref{s.asscheck}.

In the next section, we state our precise assumptions and the main result,
Theorem~\ref{t.alpha}, the proof of which comes in Section~\ref{s.proof}.
We also show in Section~\ref{sec-2}
how Corollary~\ref{cor-fixednorm} follows from Theorem~\ref{t.alpha}.
The proofs of Corollary~\ref{cor-elliptic} comes in Section~\ref
{s.asscheck}, as well as a discussion of how to estimate $\alpha$ and
the proofs of \eqref{e.alphabounds} and~\eqref{e.alphabounds2}.

\section{Setup and main results}
\label{sec-2}

\subsection{Notation and assumptions}
Throughout the paper, we work in dimension
$d\geq1$. For $r>0$ and $x\in\mathbb{R}^d$, we let $B_r(x)$ denote
the open
ball of radius $r$ centered at $x\in\mathbb{R}^d$ and denote by
$\overline{B}_r(x)$
its closure. We also set $B_r:= B_r(0)$ and $\overline{B}_r:=\overline
{B}_r(0)$.

\begin{definition}
For each $R\geq1$, we define $\mathcal M_R(\mathbb{R}^d)$ to be the
family of
centered Borel probability measures supported on $B_R$. That is, for
every $\mu\in\mathcal M(\mathbb{R}^d)$ and with $X$ the
canonical random variable on $\mathbb{R}^d$, we have
%
\begin{equation}
\label{e.centering} \E_\mu [ X ] = 0
\end{equation}
and
%
\begin{equation}
\label{e.local} \P_\mu \bigl[ \llvert X\rrvert \leq R \bigr] = 1.
\end{equation}
We also set, for each $\lambda\in(0,1]$ and $R\geq\sqrt{2d}$,
%
\begin{equation}
\label{e.ellipticity} \mathcal E_{\lambda,R}\bigl(\mathbb{R}^d\bigr):=
\bigl\{ \mu\in\mathcal M_R\bigl(\mathbb{R}^d\bigr):
\lambda I_d \leq\tfrac{1}2 \E_\mu \bigl[
XX^t \bigr] \leq I_d \bigr\}.
\end{equation}
Here, $I_d$ denotes the $d\times d$ identity matrix, and if $A$ and $B$
are symmetric matrices, then we write $A\leq B$ in the case that $B-A$
is nonnegative definite.
\end{definition}

Given a Borel
subset $\mathcal P\subseteq\mathcal{M}_R(\mathbb{R}^d)$
(the \textit{control}) and a point $x\in\mathbb{R}^d$,
we introduce the family of controlled martingales $\{(X_n,u_n)\}_{n\geq0}$
with $X_n\in\mathbb{R}^d$, $X_0=x$,
$\F_n=\sigma(X_1,\ldots,X_n)$, so that the control $u_n\in
\mathcal P$ is $\F_n$ measurable and, conditioned on
$\mathcal{F}_n$, $X_{n+1}-X_n$ is distributed according to $u_n$.
With an abuse of notation, we denote by ${\mathcal P}_n$ the class of
\textit{admissible controls}, that is, those sequences ${\mathbf
u}=(u_1,\ldots,u_n)$
satisfying the
above restrictions, and we let $\P^x$ denote the
law of the sequence $\{(X_n,u_n)\}_{n\geq0}$.
In this setup, we are interested in the evaluation, for fixed $\delta>
0$, of the quantity
%
\begin{equation}
\label{eq-control} \sup_{ \mathbf{ u}\in{\mathcal P}_n} \P^x [X_n\in
\overline B_\delta ].
\end{equation}

\begin{remark}
\label{rem-1}
It is natural to also consider the dual problem, that is, the quantity
%
\begin{equation}
\label{eq-control1} \inf_{ \mathbf{ u}\in{\mathcal P}_n} \P^x [X_n\in
\overline B_\delta ].
\end{equation}
The analysis required is similar and we comment on it in
Section~\ref{sec-moncontrol} below.
\end{remark}

We next introduce the value function, which satisfies the
dynamic programming equation.

\begin{definition}[(The value function $w$)]
\label{def.w}
Given a Borel set
$\mathcal P\subseteq\mathcal{M}_R(\mathbb{R}^d)$ and $\delta> 0$,
we define a function $w:\mathbb{R}^d\times\N\to\R$ by setting
\[
 w(x,0) := \cases{ 1, &\quad $\mbox{if } x\in\overline
B_\delta$,
\cr
0, &\quad  $\mbox{if } x\in\mathbb{R}^d\setminus
\overline B_\delta$,}
\]
and then defining $w(\cdot,n)$ inductively by
%
\begin{equation}
\label{eq-rev1} w(x,n+1) : = \sup_{\rho\in\mathcal P} \E_\rho
\bigl[ w(x+X,n) \bigr].
\end{equation}
\end{definition}

(The assumption that $\mathcal P$ is Borel ensures that the function
$w(\cdot,n)$ as determined by \eqref{eq-rev1}
can be integrated against any probability measure;
see \cite{BS}, Chapter~7.8.) It is clear that $w(x,n)$ equals the
expression in \eqref{eq-control}.

Our interest lies in the asymptotic behavior of $w(x,n)$ for large $n$.
We prove our main result under two additional assumptions, stated
below. These assumptions can be quickly checked for large classes of
examples, as we show in Section~\ref{s.asscheck}. Before stating these,
we first introduce some further notation.

\begin{definition}[(The operator $F^-$)]
Given $\mathcal P \subseteq\mathcal M_R(\mathbb{R}^d)$, we define the
operator $F^-$ on the space $C(\mathbb{R}^d)$ of continuous functions by
%
\begin{equation}
\label{e.F} F^- [ \phi ](x) := \phi(x) - \sup_{\rho\in\mathcal P}
\E_\rho \bigl[ \phi(X+x ) \bigr].
\end{equation}
We extend the definition of $F^-$ to merely locally
bounded functions $\phi$ by setting
\[
 F^- [ \phi ](x) := \phi(x) - \sup_{\rho\in\mathcal P} \sup
_{\psi\in C(\mathbb{R}^d), \psi\leq\phi} \E_\rho \bigl[ \psi(X+x ) \bigr].
\]
(The extension allows $F^-$ to act even on nonmeasurable functions, as long
as they are locally bounded.)
By abuse of notation, we also use
$F^-$ to denote the functions $\mathbb{S}^d\to\R$ (here,
$\mathbb{S}^d$ denotes nonnegative definite $d$-by-$d$ matrices)
given by
$M \mapsto
F^- [ \phi_M  ]$, where $\phi_M$ is any quadratic function with
Hessian $M \in\mathbb{S}^d$, that is, we define
\[
 F^-(M) := - \frac{1}2 \sup_{\rho\in\mathcal P}
\E_\rho [ X \cdot MX ]. 
\]
Note that, by \eqref{e.centering}, $F^-[\phi]$ is unchanged if we add an
affine function to $\phi$. In particular, for every $M\in\mathbb{S}^d$,
$p\in\mathbb{R}^d$, $a\in\R$ and quadratic $\phi(x):= \frac{1}2
x \cdot Mx + p\cdot x + a$,
we see that $F^- [ \phi ]=F^-(M)$.
\end{definition}

In general, $u \mapsto F^-(D^2u)$
is a fully nonlinear, concave, (possibly
degenerate) elliptic operator. We refer to \cite{CC,Krylov} for an
introduction to fully nonlinear elliptic PDEs. In the case that
$\mathcal P \subseteq\mathcal E_{\lambda,R}(\mathbb{R}^d)$ for some
$\lambda
> 0$,
then $F^-$ is uniformly elliptic.
If $\mathcal P = \mathcal E_{\lambda,R}(\mathbb{R}^d)$, then the
operator $F^-$
coincides with the minimal Pucci operator with
ellipticity constants $\lambda$ and $1$ (as defined, up to a sign
convention, in \cite{CC}).

As explained at the beginning of Section~\ref{s.proof}, the operator
$F^-[\cdot]$ describes the evolution of the discrete control problem
and in particular the function $w$, while $F^-(D^2\cdot)$ describes the
continuous control problem which approximates the discrete problem on
large scales (see Lemma~\ref{l.consistency}).

In the rest of the paper, except in Section~\ref{sec-moncontrol}
and Section~\ref{s.asscheck}, we write $F=F^-$. To aid our
computations, we note that, for all $t\geq0$ and locally bounded
functions $\phi,\psi: \mathbb{R}^d\to\R$,
we have
%
\begin{equation}
\label{e.poshomo} F [ t\phi ] = t F [ \phi ]
\end{equation}
and
%
\begin{equation}
\label{e.sublinear} F[\phi] + F[\psi] \leq F[\phi+\psi].
\end{equation}

We next present our two assumptions.

\begin{assumption}[($F$ admits a self-similar solution)]
\label{ass.Phi}
There exist $\alpha> 0$, $\sigma\in(0,1]$ and a solution $\Phi\in
C^2(\mathbb{R}^d\times(0,\infty))$ of the fully nonlinear (possibly
degenerate) parabolic partial differential equation
%
\begin{equation}
\label{e.Phieq} \partial_t \Phi+ F\bigl(D^2\Phi\bigr) = 0\qquad
\mbox{in } \mathbb{R}^d\times (0,\infty),
\end{equation}
which satisfies
%
\begin{eqnarray}
\label{e.positive} \Phi&>& 0 \qquad\mbox{in } \mathbb{R}^d\times(0,\infty),
\\
\label{e.selfsim} \Phi(\sqrt\lambda x,\lambda t) &=& \lambda^{-\alpha} \Phi(x,t)\qquad
\mbox {for every } \lambda> 0,
\end{eqnarray}
and $\Phi$ decays like a Gaussian up to $C^{2,\sigma}$: that is, there
exist constants $a>0$ and $K>1$ such that
%
\begin{equation}
\label{e.C2sigma} \Phi(\cdot,1) \in C^{2,\sigma} \bigl(\mathbb{R}^d
\bigr)
\end{equation}
and, for every $x\in\mathbb{R}^d$,
%
\begin{equation}
\label{e.decay} \bigl\llVert \Phi(\cdot,1) \bigr\rrVert _{C^{2,\sigma}(B_1(x))} \leq K
\exp \bigl( -a|x|^2 \bigr). %
\end{equation}
\end{assumption}

The first part of Assumption~\ref{ass.Phi}
holds in the uniformly elliptic case, $\mathcal P \subseteq
\mathcal E_{\lambda,R}(\mathbb{R}^d)$, by the results in \cite{AT}.
We have made it into an assumption with an eye toward the application
to Corollary~\ref{cor-fixednorm}. In Section~\ref{sec-4.2}, we verify that
the second part of the
assumption, namely \eqref{e.decay}, holds as well in the uniformly elliptic
case. We note that the regularity in \eqref{e.decay}
is key in relating the discrete and continuous control problems.

The next assumption is specific to the discrete setup and allows to control
the behavior of the discrete control problem before convergence to the
continuous problem is achieved.

\begin{assumption}[(Behavior of $w$ up to finite times)]
\label{ass.w}
We have:
\begin{longlist}[(ii)]
\item[(i)] For every $r>0$, there exists $N_0(r)>1$ such that, for
every $N\geq N_0(r)$,
%
\begin{equation}
\label{e.assbulk} \inf \bigl\{ w(x,N) : |x| \leq r\sqrt{N \log N} \bigr\} >0.
\end{equation}
\item[(ii)] For every $r>1$ and $n\in\N$, there exists $C( r,n)>1$ such that
%
\begin{equation}
\label{e.asstails} w(x,n) \leq C \exp \biggl( -r \frac{|x|}{\sqrt{n}} \biggr)\qquad \mbox
{for every } x\in\mathbb{R}^d.
\end{equation}
\end{longlist}
\end{assumption}

Note that, in our setup, \eqref{e.asstails} is always satisfied, due to
Azuma's inequality \cite{azuma}.
Indeed, for any (multidimensional)
martingale $\{W_n\}$ with bounded increments,
\[
\P\bigl[|W_n+x|\leq1\bigr]\leq\P \biggl[\biggl| W_n \cdot
\frac{x}{|x|}\biggr|\geq\bigl ||x|-1\bigr| \biggr] \leq\exp
\bigl({- c\bigl(|x|-1\bigr)^2/n}
\bigr) ,
\]
where the last inequality follows from Azuma's inequality
using the fact that $ W_n \cdot x/|x|$ is a one-dimensional martingale
with bounded increments. Dealing separately with the case $|x|<\sqrt{n}+1$
and $|x|>\sqrt{n}-1$ and adjusting the constant $C(r,n)$ yields \eqref
{e.asstails}.

\subsection{Local asymptotics for $w$}
We next present the main result of the paper.

\begin{theorem}
\label{t.alpha}
\textup{(i)} Assume that Assumptions \ref{ass.Phi} and \ref{ass.w}\textup{(ii)} hold. Then
%
\begin{equation}
\label{e.theorem1-UB} \sup_{r>0} \limsup_{n\to\infty} \sup
_{x\in B_{r\sqrt{n}}} \frac{
w(x,n)}{\Phi(x,n)} < +\infty.
\end{equation}

\textup{(ii)} Assume that Assumptions \ref{ass.Phi} and \ref{ass.w}\textup{(i)} hold. Then
%
\begin{equation}
\label{e.theorem1-LB} 0 < \inf_{r>0} \liminf_{n\to\infty}
\inf_{x\in B_{r\sqrt{n}}} \frac
{w(x,n)}{\Phi(x,n)}.
\end{equation}
\end{theorem}

Observe that, in view of \eqref{e.positive} and \eqref{e.selfsim}, the
inequalities  \eqref{e.theorem1-UB} and \eqref{e.theorem1-LB} together
imply that, for every $r>0$,
\[
 0< \liminf_{n\to\infty} \inf_{x\in B_{r\sqrt{n}}}
n^\alpha w(x,n) \leq \limsup_{n\to\infty} \sup
_{x\in B_{r\sqrt{n}}} n^\alpha w(x,n) < \infty.
\]
These can be compared to the conclusions of Corollaries \ref
{cor-elliptic} and \ref{cor-fixednorm}.

Given the conclusion of Theorem~\ref{t.alpha}, it is natural to expect
a stronger statement to hold, namely a full local limit theorem for
$w$: that is, for some constant $L>0$,
%
\begin{equation}
\label{e.locallimit} \sup_{r>0} \limsup_{n\to\infty} \sup
_{x\in B_{r\sqrt{n}}} \biggl\llvert \frac
{ w(x,n)}{\Phi(x,n)} -L \biggr\rrvert = 0.
\end{equation}
While our setup may be a bit too general for \eqref{e.locallimit}, we
do expect it to hold, for instance, in the uniformly elliptic setting
($\mathcal P \subseteq\mathcal E_{\lambda,R}$). Indeed, this is
relatively easy to obtain from Theorem~\ref{t.alpha} and the test
functions in Section~\ref{s.proof}, provided we have at our disposal
some regularity theory for uniformly parabolic finite difference
equations (which we would apply to $w$). We could not find such a
result matching our situation. We speculate that one could derive it
from adaptations of known techniques, however developing such a
regularity theory would take us too far astray from the focus of this
paper, and so we do not prove \eqref{e.locallimit}.

To further motivate our assumptions, we bring now the proof
of Corollary~\ref{cor-fixednorm}.
That is, we consider the particular case of martingales
whose increments are bounded with norm of constant second moment.
\begin{pf*}{Proof of Corollary~\ref{cor-fixednorm}}
For some $\lambda\in(0,1]$, we set
\[
 \mathcal P : = \bigl\{ \rho\in\mathcal M_1\bigl(
\mathbb{R}^d\bigr) : \E_\rho \bigl[ |X|^2
\bigr] = \lambda \bigr\}.
\]
We easily check that the operator $F^-$ can be expressed by
\[
 F^-(M) = -\tfrac{1}2\lambda\cdot\mbox{(largest eigenvalue
of $M$)}.
\]
By a direct computation,
Assumption~\ref{ass.Phi} holds for $F=F^-$ with $\alpha= \frac{1}2$ and
\[
 \Phi(x,t) := t^{-{1}/2} \exp \biggl( -\frac{|x|^2}{2\lambda t}
\biggr).
\]
As noted before,
Azuma's inequality implies that Assumption~\ref{ass.w}(ii) holds.
We therefore obtain Corollary~\ref{cor-fixednorm} as a consequence of
Theorem~\ref{t.alpha}(i).
\end{pf*}

\subsection{Minimal probabilities}
\label{sec-moncontrol}
As discussed in Remark~\ref{rem-1}, it is natural to consider the
optimal control problem \eqref{eq-control1} instead of
\eqref{eq-control}, with associated value function $v$ satisfying the
dynamic programming equation
\[
v(x,n+1)=\inf_{\rho\in\mathcal P}\E_\rho v(x+X,n),\qquad v(x,0)=w(x,0).
\]
The analysis is similar,
with the operator $F^-$
replaced by the operator
\[
 F^+ [ \phi ](x) := \phi(x) - \inf_{\rho\in\mathcal P} \inf
_{\psi\in C(\mathbb{R}^d), \psi\geq\phi} \E_\rho \bigl[ \psi(X+x ) \bigr] ,
\]
and a similar definition for $F^+(M), M\in\mathbb{S}^d$. In the
analysis, the
relations
\[
F^+[-\phi]=-F^{-}[\phi],\qquad F^+[\phi]\geq F^-[\phi]
\]
and
\[
F^+[\phi]+F^-[\psi]\leq F^+[\phi+\psi]\leq F^+[\phi]+F^+[\phi]
\]
come in handy. Using now $F=F^+$ and
replacing $w$ by $v$ in Assumption~\ref{ass.w}, one then obtains
Theorem~\ref{t.alpha} for $v$.

We remark that, in some natural situations,
the assumption \eqref{e.assbulk} holds
for $w$ but not for $v$. For an example,
see the
setup of Corollary~\ref{cor-fixednorm} with $d\geq2$.
In that situation, one can use,
when at $x\neq0$,
controls in a direction tangential
to the sphere centered at the origin and passing through $x$,
to conclude
that $v(x,n)=0$ if
$|x|>2\sqrt{d}\delta$ and
$n\geq0$. Thus, Assumption~\ref{ass.w}(i) does not hold for $v$.


\section{Proof of Theorem \texorpdfstring{\protect\ref{t.alpha}}{2.7}}
\label{s.proof}

In this section, we prove the local limit theorem for the value
function $w$. We proceed by presenting some lemmas needed in the
argument, beginning with some basic properties of the finite difference
equation.
Throughout, we assume that $\mathcal{P}$ is a fixed Borel subset of
$\mathcal{M}_R(\R^d)$ for some fixed $R\geq\sqrt{2d}$.

Recall that the equation satisfied by $w$ is
%
\begin{equation}
\label{e.FDsimple} w(x,n+1) = \sup_{\rho\in\mathcal P} \E_\rho \bigl[
w(x+X,n) \bigr].
\end{equation}
It can be written in the equivalent form
%
\begin{equation}
\label{e.FD} w(x,n+1) - w(x,n) + F \bigl[ w(\cdot,n) \bigr](x) = 0,
\end{equation}
which is an explicit finite difference scheme for the (continuum)
parabolic equation
%
\begin{equation}
\label{e.cont} w_t + F\bigl(D^2w\bigr) = 0.
\end{equation}
We first record the fact that the scheme is in fact \emph{consistent}
with \eqref{e.cont}.

\begin{lemma}
\label{l.consistency}
There exists $C(R)>0$ such that, for every $\sigma\in(0,1]$, $\varphi
\in C^{2,\sigma}(\mathbb{R}^d)$ and $x\in\mathbb{R}^d$,
\[
 \bigl\llvert F [ \varphi ](x) - F\bigl(D^2\varphi(x)
\bigr) \bigr\rrvert \leq C \bigl[ D^2 \varphi \bigr]_{C^{\sigma}(B_R(x))}.
\]
\end{lemma}
\begin{pf}
It is enough to consider the case $x=0$.
We have
\[
-F [ \varphi ](0) + F\bigl(D^2\varphi(0)\bigr) = \sup
_{\rho\in\mathcal P} \E_\rho \bigl[ \varphi(X) - \varphi (0) \bigr]
- \frac{1}2 \sup_{\rho\in\mathcal P} \E_\rho \bigl[ X\cdot
D^2\varphi(0) X \bigr].
\]
%
This implies
%
\[
 \bigl\llvert F [ \varphi ](0) - F\bigl(D^2\varphi(0)
\bigr) \bigr\rrvert \leq\sup_{\rho\in\mathcal P} \biggl\llvert
\E_\rho \biggl[ \varphi(X) - \varphi(0) - \frac{1}2 X\cdot
D^2\varphi (0) X \biggr] \biggr\rrvert .
\]
Using the centering condition and then Taylor's formula, we find that,
for any $\rho\in\mathcal P$,
\begin{eqnarray*}
&&\biggl| \E_\rho \biggl[ \varphi(X) - \varphi(0) -
\frac{1}2 X\cdot D^2\varphi(0) X \biggr] \biggr|
\\
& &\qquad\leq \E_\rho \biggl[ \biggl| \varphi(X) - \varphi(0) - X \cdot D
\varphi(0) - \frac{1}2 X\cdot D^2\varphi(0) X \biggr| \biggr]
\\
&&\qquad \leq\sup_{y\in\overline B_R} \biggl| \varphi(y) - \varphi(0) - y \cdot D
\varphi(0) - \frac{1}2 y\cdot D^2\varphi(0) y \biggr|
\\
&&\qquad \leq C \bigl[ D^2 \varphi \bigr]_{C^\sigma(B_R)}.
\end{eqnarray*}
Note that we used both \eqref{e.centering} and \eqref{e.local} in the
third line and then Taylor's formula in the last line above.
\end{pf}

We next check that the finite difference scheme \eqref{e.FD} is \emph
{monotone}, that is, it satisfies a comparison principle.

\begin{lemma}
\label{l.monotonicity}
Assume $u, v:\mathbb{R}^d\to\R$ are locally bounded and
satisfy, for each $x\in\mathbb{R}^d$ and $n\in\N$,
\[
 \cases{ u(x,n+1) - u(x,n) + F \bigl[ u(\cdot,n) \bigr](x) \leq0,
\vspace*{2pt}\cr
v(x,n+1) - v(x,n) + F \bigl[ v(\cdot,n) \bigr](x) \geq0,
\vspace*{2pt}\cr
u(x,0) \leq
v(x,0).}
\]
Then $u \leq v$ in $\mathbb{R}^d\times\N$.
\end{lemma}
\begin{pf}
Using the form \eqref{e.FDsimple} rather than \eqref{e.FD}, we observe
that, for every $x\in\mathbb{R}^d$,
\begin{eqnarray*}
 u(x,1) &\leq&\sup_{\rho\in\mathcal P} \E_\rho \bigl[
u(x+X,0) \bigr] \\
&\leq&\sup_{\rho\in\mathcal P} \E_\rho \bigl[
v(x+X,0) \bigr] \leq v(x,1).
\end{eqnarray*}
The lemma now follows by induction.
\end{pf}

The proof of Theorem~\ref{t.alpha} requires a test function
calculation, similar to the one in \cite{AT}, Lemma~4.4. The result is
summarized in the following lemma.

\begin{lemma}
\label{l.Psi}
Fix $\beta> 0$ and consider the function
%
\begin{equation}
\label{e.Psi} \Psi(x,t) := t^{-\beta} \exp \biggl( -\beta \biggl( 1+
\frac{|x|^2}t \biggr)^{1/2} \biggr).
\end{equation}
Then there exist $C(d,R,\beta)>1$ and $c(d,R,\beta)>0$ such that, for
every $x\in\mathbb{R}^d$ and $t\geq C$,
%
\begin{eqnarray}
\label{e.Psieq} &&\Psi(x,t+1) - \Psi(x,t) + F \bigl[ \Psi(\cdot,t) \bigr] (x)
\nonumber
\\[-8pt]
\\[-8pt]
\nonumber
&&\qquad\geq
c t^{-1} \Psi(x,t) \cdot\cases{ -C, & \quad $\mbox{if } |x| \leq C \sqrt t$,
\vspace*{2pt}
\cr
\displaystyle\frac{|x|}{\sqrt t}, &\quad  $\mbox{if } |x| \geq C \sqrt t$.}
\end{eqnarray}
\end{lemma}
\begin{pf}
We split the computation into three steps: first we estimate $\partial
_t \Psi+ F^-(D^2\Psi)$ from below and in the last two steps we show by
approximation that this cannot be too much different from the finite
difference scheme. Throughout, $C$ and $c$ denote positive constants
which depend only on $(d,R,\beta)$ and may vary in each occurrence.

\emph{Step} 1. We estimate $\partial_t \Psi+F(D^2\Psi)$ from below.
We compute
%
\begin{eqnarray}
\label{e.Psit} \partial_t \Psi(x,t)& =& -\beta t^{-1}
\Psi(x,t) \biggl(1 - \biggl(1 + \frac
{|x|^2}{t} \biggr)^{-1/2}
\frac{|x|^2}{2t} \biggr),
\\
\label{e.DPsi} D\Psi(x,t)& = &- t^{-1/2} \Psi(x,t) \biggl( \beta \biggl(
1+\frac{|x|^2}{t} \biggr)^{-1/2} \biggr)\frac{x}{\sqrt{t}}
\end{eqnarray}
and
%
\begin{eqnarray}
\label{e.D2Psi} D^2 \Psi(x,t) &=& -\beta t^{-1} \Psi(x,t)
\biggl( \biggl( 1 + \frac
{|x|^2}{t} \biggr)^{-1/2} I_d -
\beta \biggl( 1 + \frac
{|x|^2}{t} \biggr)^{-1} \frac{x\otimes x}{t}
\nonumber
\\[-8pt]
\\[-8pt]
\nonumber
&&{}- \biggl( 1 + \frac{|x|^2}{t} \biggr)^{-3/2} \frac{x\otimes x}{t}
\biggr).
\end{eqnarray}
Using that
\[
 \frac{|x|^2}{t} I \geq\frac{x\otimes x}{t},
\]
we may discard the first and third terms in parentheses to obtain
\[
 D^2 \Psi(x,t) \leq\beta^2
t^{-1} \biggl( \biggl( 1 + \frac
{|x|^2}{t} \biggr)^{-1}
\frac{x\otimes x}{t} \biggr) \Psi(x,t).
\]
Inserting this expression into the operator $F$ and using \eqref
{e.poshomo} and \eqref{e.sublinear}, we obtain
\begin{eqnarray*}
 F\bigl(D^2\Psi(x,t)\bigr)& \geq&-R^2
\beta^2 t^{-1} \biggl( 1 + \frac{|x|^2}{t}
\biggr)^{-1} \frac{|x|^2}{t} \Psi(x,t) \\
&\geq&-\beta^2
t^{-1} \Psi(x,t).
\end{eqnarray*}
It follows that
\begin{eqnarray*}
\partial_t \Psi(x,t) + F\bigl(D^2\Psi(x,t) \bigr) & \geq&
\beta t^{-1} \Psi(x,t) \biggl( \biggl( 1 + \frac{|x|^2}{t}
\biggr)^{-1/2} \frac{|x|^2}{2t}-\bigl(1 + R^2 \beta\bigr)
\biggr)
\\
& \geq&\beta t^{-1} \Psi(x,t) \biggl( \biggl( 1 + \frac
{|x|^2}{t}
\biggr)^{-1/2} \frac{|x|^2}{2t}-C \biggr)
\\
& \geq& c t^{-1} \Psi(x,t) \cdot\cases{ -C, &\quad  $\mbox{if } |x| \leq C
\sqrt t$, \vspace*{2pt}
\cr
\displaystyle\frac{|x|}{\sqrt t}, &\quad  $\mbox{if } |x| \geq C \sqrt t$.}
\end{eqnarray*}

\emph{Step} 2. In preparation to evaluate $\Psi$ on the finite
difference scheme by comparing to step 1, we estimate $|D^3\Psi|$ and
$\partial_t^2 \Psi$. The claims are: for all $x\in\mathbb{R}^d$ and
$t\geq1$,
%
\begin{equation}
\label{e.uglyduck1} \bigl\llvert \partial^2_t \Psi(x,t) \bigr
\rrvert \leq C t^{-2} \biggl( 1 + \frac
{|x|}{\sqrt t} \biggr) \Psi(x,t)
\end{equation}
and
%
\begin{equation}
\label{e.uglyduck2} \bigl\llvert D^3 \Psi(x,t) \bigr\rrvert \leq C
t^{-1} \Psi(x,t) + Ct^{-3/2} \biggl(1+\frac{|x|}{\sqrt t}
\biggr) \Psi(x,t).
\end{equation}
Differentiating \eqref{e.Psit} yields
\begin{eqnarray*}
 \partial_t^2 \Psi(x,t)& =& \beta
t^{-2} \Psi(x,t) \biggl( 1+ \biggl(1+\frac
{|x|^2}{t}
\biggr)^{-1/2} \frac{|x|^2}{t} + \biggl(1+\frac{|x|^2}{t}
\biggr)^{-3/2} \frac{|x|^4}{t^2}
\\
&&{}+\beta \biggl( 1- \biggl(1+\frac
{|x|^2}{t} \biggr)^{-1/2}
\frac{|x|^2}{t} \biggr) \biggr),
\end{eqnarray*}
from which we get \eqref{e.uglyduck1}. To prove \eqref{e.uglyduck2}, we
must differentiate \eqref{e.D2Psi}. Define $M(x,t)$ to be the matrix in
the parentheses in \eqref{e.D2Psi}, so that
\[
\bigl |D^3\Psi(x,t)\bigr| \leq Ct^{-1} \Psi(x,t)
\bigl|DM(x,t)\bigr| + Ct^{-1} \bigl|D\Psi (x,t)\bigr| \bigl|M(x,t)\bigr|.
\]
It is easy to check that, for $x\in\mathbb{R}^d$ and $t\geq1$,
\[
 \bigl|M(x,t)\bigr| + \bigl|DM(x,t)\bigr| \leq C.
\]
Using this and \eqref{e.DPsi}, we obtain \eqref{e.uglyduck2}.

\emph{Step} 3. We evaluate $\Psi$ on the finite difference scheme.
From \eqref{e.uglyduck1} we have, for every $(x,t) \in\mathbb
{R}^d\times
(1,\infty)$,
\[
 \Psi(x,t+1) - \Psi(x,t) \geq\partial_t\Psi(x,t) -C
t^{-2} \biggl( 1 + \frac{|x|}{\sqrt t} \biggr) \Psi(x,t)
\]
and, by Lemma~\ref{l.consistency} and \eqref{e.uglyduck2},
\[
 F \bigl[ \Psi(\cdot,t) \bigr] (x) \geq F\bigl(D^2
\Psi(x,t) \bigr) - C t^{-1} \Psi (x,t) - Ct^{-3/2} \biggl(1+
\frac{|x|}{\sqrt t} \biggr) \Psi(x,t).
\]
Putting these together, we finally obtain that, for every $x\in\mathbb{R}^d$
and $t\geq C$
\begin{eqnarray*}
 &&\Psi(x,t+1) - \Psi(x,t) + F \bigl[ \Psi(\cdot,t) \bigr]
(x)
\\
&&\qquad \geq\partial_t\Psi(x,t) + F\bigl(D^2 \Psi(x,t) \bigr)
- C t^{-1} \Psi (x,t)- C t^{-3/2} \Psi(x,t) \frac{|x|}{\sqrt t}
\\
&&\qquad \geq c t^{-1} \Psi(x,t) \cdot\cases{ -C, &\quad  $\mbox{if } |x| \leq C
\sqrt t$,
\cr
\displaystyle\frac{|x|}{\sqrt t}, &\quad $ \mbox{if } |x| \geq C \sqrt t$.}
\end{eqnarray*}
This is \eqref{e.Psieq}.
\end{pf}

To prepare for the proof of Theorem~\ref{t.alpha}, we must perform a
second computation to show that, up to a suitable error, $\Phi$ is a
solution of the finite difference equation. In fact, we bend $\Phi$
slightly in order to make it a strict subsolution or supersolution
of \eqref{e.FD} in the region $|x| \lesssim\sqrt{t}$. This computation
is summarized in the following two lemmas.

\begin{lemma}
\label{l.Phitheta}
Let Assumption~\ref{ass.Phi} hold. For each $\theta> 0$, define
\[
 \Phi_\theta(x,t) := \exp \biggl( \frac{1}\theta
t^{-\theta} \biggr) \Phi(x,t).
\]
Then $\Phi_\theta$ satisfies, for some $C(d,R,\theta,\sigma
,a,\alpha)>1$,
%
\begin{eqnarray}\qquad
\label{e.Phitheta}&& \Phi_\theta(x,t+1) - \Phi_\theta(x,t) + F
\bigl[ \Phi_\theta (\cdot,t) \bigr](x)
\nonumber
\\[-8pt]
\\[-8pt]
\nonumber
&&\qquad\leq-t^{-1-\theta} \Phi(x,t) + C t^{-1-\alpha-\sigma/2} \exp \biggl( -
\frac{a|x|^2}{2t} \biggr) \qquad\mbox{in } \mathbb {R}^d\times (0,\infty).
\end{eqnarray}
\end{lemma}

\begin{lemma}
\label{l.Phitheta2}
Let Assumption~\ref{ass.Phi} hold. For each $\theta> 0$, define
\[
 \Phi_{-\theta}(x,t) := \exp \biggl( -\frac{1}\theta
t^{-\theta} \biggr) \Phi(x,t).
\]
Then $\Phi_{-\theta}$ satisfies, for some $C(d,R,\theta,\sigma
,a,\alpha)>1$,
%
\begin{eqnarray}\qquad
\label{e.Phitheta2} &&\Phi_{-\theta} (x,t+1) - \Phi_{-\theta} (x,t) + F
\bigl[ \Phi _{-\theta
}(\cdot,t) \bigr](x)
\nonumber
\\[-8pt]
\\[-8pt]
\nonumber
&&\qquad\geq t^{-1-\theta} \Phi(x,t) - C t^{-1-\alpha
-\sigma/2} \exp \biggl( -
\frac{a|x|^2}{2t} \biggr) \qquad\mbox{in } \mathbb{R}^d \times(0,\infty).
\end{eqnarray}
\end{lemma}

\begin{pf*}{Proof of Lemma~\ref{l.Phitheta}}
Similar to the proof of Lemma~\ref{l.Psi}, we first insert $\Phi
_\theta
$ into the continuum equation, estimate this from above, and then
transfer the estimate by approximation to the finite difference
equation. Throughout, $C$ and $c$ denote positive constants which may
vary in each occurrence and depend only on $(d,R,\theta,\sigma
,a,\alpha)$.

\emph{Step} 1. We evaluate $\partial_t \Phi_\theta+ F(D^2\Phi
_\theta
)$. We compute
%
\begin{equation}
\label{e.ptPhitheta} \partial_t \Phi_\theta(x,t) = \exp \biggl(
\frac{1}\theta t^{-\theta} \biggr) \bigl( -t^{-1-\theta}
\Phi(x,t) + \partial_t \Phi(x,t) \bigr)
\end{equation}
and
\[
 D^2 \Phi_\theta= \exp \biggl(
\frac{1}\theta t^{-\theta} \biggr) D^2 \Phi(x,t).
\]
Using \eqref{e.poshomo}, we find, for every $x\in\mathbb{R}^d$ and
$t> 0$,
%
\begin{eqnarray}
\label{e.Phithetacont}&& \partial_t \Phi_\theta(x,t) + F
\bigl(D^2\Phi_\theta(x,t)\bigr)
\nonumber
\\[-8pt]
\\[-8pt]
\nonumber
&&\qquad= -\exp \biggl(
\frac{1}\theta t^{-\theta} \biggr) t^{-1-\theta} \Phi(x,t) \leq
-t^{-1-\theta} \Phi(x,t).
\end{eqnarray}

\emph{Step} 2. We estimate the quantity $ [ D^2 \Phi_\theta
(\cdot
,t)  ]_{C^{0,\sigma}B_R(x))}$. Assumptions \eqref{e.selfsim}
and \eqref{e.C2sigma} imply that, for every $(x,t) \in\mathbb
{R}^d\times
(0,\infty)$,
\begin{eqnarray*}
 \bigl[ D^2\Phi(\cdot,t) \bigr]_{C^{0,\sigma}(B_{R}( x ))} & =&
t^{-1-\alpha-\sigma/2} \bigl[ D^2\Phi(\cdot,1) \bigr]_{C^{0,\sigma
}(B_R(x/\sqrt{t}))}
\\
& \leq& C t^{-1-\alpha-\sigma/2} \exp \biggl( - \frac{a|x|^2}{t} \biggr)
\end{eqnarray*}
and, therefore, for every $(x,t) \in\mathbb{R}^d\times(1,\infty)$,
%
\begin{eqnarray}
\label{e.D2sPhitheta} \bigl[ D^2\Phi_\theta(\cdot,t)
\bigr]_{C^{0,\sigma}(B_{R}( x ))} & =& \exp \biggl( \frac{1}\theta t^{-\theta}
\biggr) \bigl[ D^2\Phi(\cdot,t) \bigr]_{C^{0,\sigma}(B_{R}( x ))}
\nonumber
\\[-8pt]
\\[-8pt]
\nonumber
& \leq& C t^{-1-\alpha-\sigma/2}\exp \biggl( - \frac{a|x|^2}{t} \biggr).
\end{eqnarray}
%

\emph{Step} 3. We estimate the quantity $\llvert  \partial^2_t \Phi
_\theta\rrvert  $. The claim is
%
\begin{equation}
\label{e.partl2clm} \partial^2_t \Phi_\theta(x,t)
\leq C t^{-2-\alpha} \exp \biggl( -\frac
{a|x|^2}{t} \biggr).
\end{equation}
%
It is convenient to use self-similarity \eqref{e.selfsim} to relate the
time differences to spatial ones, in view of assumption \eqref
{e.decay}. First, differentiating the self-similarity relation yields
\[
 \partial_t \Phi(x,t) = -\tfrac{1}2
t^{-1} x\cdot D\Phi(x,t) - \alpha t^{-1} \Phi(x,t)
\]
and
\begin{eqnarray*}
 \partial_t^2 \Phi(x,t) &=& \alpha(\alpha+1)
t^{-2} \Phi(x,t) + \bigl( \alpha+ \tfrac{3}4 \bigr)
t^{-2} x\cdot D\Phi(x,t)\\
&&{} + \tfrac{1}4 t^{-2} x\cdot
D^2\Phi(x,t)x.
\end{eqnarray*}
Using \eqref{e.selfsim} again and then \eqref{e.decay}, we estimate
\begin{eqnarray*}
\bigl\llvert \partial_t \Phi(x,t) \bigr\rrvert & \leq& C
t^{-1-\alpha} \biggl( \Phi \biggl( \frac{x}{\sqrt t}, 1 \biggr) +
\frac{|x|}{\sqrt t} \biggl\llvert D\Phi \biggl(\frac{x}{\sqrt t},1 \biggr) \biggr
\rrvert \biggr)
\\
& \leq& Ct^{-1-\alpha} \biggl( 1+\frac{|x|}{\sqrt t} \biggr) \exp \biggl( -
\frac{a |x|^2}{t} \biggr)
\\
& \leq& Ct^{-1-\alpha} \exp \biggl( - \frac{a|x|^2}{2 t} \biggr)
\end{eqnarray*}
and
\begin{eqnarray*}
\bigl\llvert \partial_t^2 \Phi(x,t) \bigr\rrvert & \leq&
C t^{-2-\alpha} \biggl( \Phi \biggl( \frac{x}{\sqrt t}, 1 \biggr) +
\frac{|x|}{\sqrt t} \biggl\llvert D\Phi \biggl(\frac{x}{\sqrt t},1 \biggr) \biggr
\rrvert + \frac{|x|^2}{t} \biggl\llvert D^2\Phi \biggl(
\frac{x}{\sqrt t},1 \biggr) \biggr\rrvert \biggr)
\\
& \leq& Ct^{-2-\alpha} \biggl( 1+\frac{|x|^2}{t} \biggr) \exp \biggl( -
\frac{a |x|^2}{t} \biggr)
\\
& \leq& Ct^{-2-\alpha} \exp \biggl( - \frac{a|x|^2}{2 t} \biggr).
\end{eqnarray*}
In view of the fact that
\[
 \partial^2_t \Phi_\theta(x,t) =
\exp \biggl( \frac{1}\theta t^{-\theta} \biggr) \bigl(
\partial^2_t \Phi(x,t) -2t^{-1-\theta}
\partial_t \Phi (x,t) + (1+\theta)t^{-2-\theta} \Phi(x,t) \bigr),
\]
we obtain
\begin{eqnarray*}
\partial^2_t \Phi_\theta(x,t) & \leq & C \bigl(
\bigl\llvert \partial^2_t \Phi (x,t) \bigr\rrvert +
t^{-1-\theta} \bigl\llvert \partial_t \Phi(x,t) \bigr\rrvert +
t^{-2-\theta} \Phi(x,t) \bigr)
\\
& \leq& C t^{-2-\alpha} \exp \biggl( -\frac{a|x|^2}{t} \biggr).
\end{eqnarray*}
This is \eqref{e.partl2clm}.

\emph{Step} 4. We complete the proof using Lemma~\ref{l.consistency}
combined with \eqref{e.Phithetacont}, \eqref{e.D2sPhitheta}
and \eqref
{e.partl2clm}. We have
\begin{eqnarray*}
&& \Phi_\theta(x,t+1) - \Phi_\theta(x,t) + F \bigl[
\Phi _\theta (\cdot,t) \bigr](x)
\\
& &\qquad\leq\partial_t \Phi(x,t) + F \bigl( D^2\Phi(x,t)
\bigr) - Ct^{-2-\alpha} \exp \biggl( -\frac{a|x|^2}{2t} \biggr)
\\
&&\qquad\quad{} - C t^{-1-\alpha-\sigma/2} \exp \biggl(-\frac{a|x|^2}{t} \biggr)
\\
& &\qquad \leq-t^{-1-\theta} \Phi(x,t) + C t^{-1-\alpha-\sigma/2} \exp \biggl( -
\frac{a|x|^2}{2t} \biggr),
\end{eqnarray*}
as desired.
\end{pf*}

\begin{pf*}{Proof of Lemma~\ref{l.Phitheta2}}
The proof is essentially the same as that of Lem\-ma~\ref{l.Phitheta},
with only minor modifications coming from the change in sign of $\theta
$ in the definition of $\Phi_{-\theta}$. The details are omitted.
\end{pf*}

Using the above Lemmas \ref{l.Phitheta} and
\ref{l.Phitheta2} together with
Assumptions \ref{ass.Phi} and \ref{ass.w}, we now give the proof of
Theorem~\ref{t.alpha}.

\begin{pf*}{Proof of \eqref{e.theorem1-LB} \normalfont{(lower bound)}}
The proof is based on the fact that after a long time, and for
appropriate choices of the parameters, the function
%
\begin{equation}
\label{e.zeta} \zeta(x,t) := \Phi_\theta(x,t) - s\Psi(x,t)
\end{equation}
is a subsolution of the finite difference equation. Here, $\Psi$
and $\Phi_\theta$ are as in Lemmas~\ref{l.Psi} and \ref{l.Phitheta},
respectively, and $s>1$ is a large constant to be selected below. Once
we show this, the lower bound \eqref{e.theorem1-LB} follows easily from
Lemma~\ref{l.monotonicity}.

Throughout the proof, $C$ and $c$ denote positive constants
which may vary in each occurrence and depend only on $(d,R,\sigma
,a,\alpha,\delta)$.
When constants depend on other parameters, we will denote it in the notation,
for example, a constant depending on the above parameters \textit{and}
on $r$ will be denoted $C(r)$.

We first fix the parameters (with the exception of $s$). With $\alpha
>0$ and $\sigma>0$ as in Assumption~\ref{ass.Phi}, we first select
$\theta>0$ in the definition of $\Phi_\theta$ such that
%
\begin{equation}
\label{e.theta} 0 < \theta< \frac{\sigma}2.
\end{equation}
We then take $\beta> 0$ in the definition of $\Psi$ to satisfy
%
\begin{equation}
\label{e.beta} \alpha+\theta< \beta< \alpha+ \frac{\sigma}2.
\end{equation}

\emph{Step} 1. We show that there exists $T(s)>1$ such that $\zeta$
defined as in \eqref{e.zeta} satisfies, for every $x\in\mathbb{R}^d$ and
$t\geq T$,
%
\begin{equation}
\label{e.zetasub} \zeta(x,t+1) - \zeta(x,t) + F \bigl[ \zeta(\cdot,t) \bigr] (x)
\leq0.
\end{equation}
According to \eqref{e.poshomo} and \eqref{e.sublinear}, for every $t>0$,
\[
 F \bigl[ \zeta(\cdot,t) \bigr](x) \leq F \bigl[
\Phi_\theta(\cdot,t) \bigr](x) - sF \bigl[ \Psi(\cdot,t) \bigr](x).
\]
Then according to \eqref{e.Psieq}, \eqref{e.Phitheta}
and Assumption~\ref{ass.Phi},
%
\begin{eqnarray}
\label{e.zetaeq}&& \zeta(x,t+1) - \zeta(x,t) + F \bigl[ \zeta(\cdot,t)
\bigr] (x)\nonumber
\\
&&\qquad \leq-t^{-1-\alpha-\theta} \Phi \biggl( \frac{x}{\sqrt t},1 \biggr) +
Ct^{-1-\alpha-\sigma/2} \exp \biggl( -\frac{a|x|^2}{2t} \biggr)
\\
&&\qquad\quad{}
- cs t^{-1-\beta} \Psi \biggl( \frac{x}{\sqrt t},1 \biggr) \cdot
\cases{ -C, &\quad $ \mbox{if } |x| \leq C \sqrt{t},$
\vspace*{2pt}\cr
\displaystyle\frac{|x|} {\sqrt t}, &\quad $ \mbox{if }
|x| \geq C\sqrt{t}$.}\nonumber
\end{eqnarray}
Now we simply note that if $t$ is large enough, then by \eqref{e.theta}
and \eqref{e.beta} we have, for every $|y| \leq C$,
%
\begin{equation}\qquad
\label{e.slowdecaywins} t^{-1-\alpha-\theta} \Phi(y,1) \geq C t^{-1-\alpha-\sigma/2} \exp \biggl(-
\frac{a|y|^2}{2} \biggr) + Cst^{-1-\beta} \Psi (y,1 ),
\end{equation}
while on the other hand, for every $|y| \geq C$, condition \eqref
{e.beta} and the fact that $\Psi(\cdot,1)$ has fatter tails than a
Gaussian ensures that, for large enough $t$,
%
\begin{equation}
\label{e.fattailswin} |y| t^{-1-\beta} \Psi( y,1) \geq Ct^{-1-\alpha-\sigma/2} \exp
\biggl( -\frac{a|y|^2}{2} \biggr).
\end{equation}
These inequalities imply that, for large enough $t$, the right-hand
side of \eqref{e.zetaeq} is nonpositive in $\mathbb{R}^d$, as
claimed. In terms
of $s$, we see that it
suffices to take
%
\begin{equation}
\label{e.Ts} T(s) := C s^{1/(\beta-\alpha-\theta)},
\end{equation}
where $C$, according to our convention, is a constant that depends
on $(d,R,\sigma,a,\alpha)$ only (in particular,
it does not depend on $s$).

\emph{Step} 2. We complete the proof of the lower bound.
Denote $N:=\lceil T \rceil$, and observe that, in view of \eqref{e.Ts},
\begin{eqnarray*}
 \bigl\{ y\in\mathbb{R}^d: \zeta(y,N) \geq0 \bigr\} &
= &\bigl\{ y\in\mathbb {R}^d: \Phi _\theta(y,N) \geq s
\Psi(y,N) \bigr\}
\\
& \subseteq& \biggl\{ y\in\mathbb{R}^d: N^{-\alpha} \exp \biggl(
- \frac
{a|y|^2}{N} \biggr) \geq cs N^{-\beta} \exp \biggl( -
\frac{\beta
|y|}{\sqrt N} \biggr) \biggr\}
\\
& \subseteq& \biggl\{ y\in\mathbb{R}^d: \exp \biggl( -
\frac
{a|y|^2}{N} \biggr) \geq c N^{-\theta} \exp \biggl( -
\frac{\beta|y|}{\sqrt N} \biggr) \biggr\}
\\
& \subseteq& \bigl\{ y\in\mathbb{R}^d: |y| \leq C \sqrt{N \log N}
\bigr\}.
\end{eqnarray*}
(In the second inclusion, we used that $s\geq N^{\beta-\alpha-\theta}$.)

Take $r$ as equal to $C$ of the last display (recall that $C$ does not
depend on $s$). Let $N_0(r)$ be as in part (i) of
Assumption~\ref{ass.w}, and choose $s$ large enough so that
$N>N_0(r)$. Note that, given the function $N_0(\cdot)$,
$s=s(d,R,\sigma,a,\alpha)$, and thus $N=N(d,R,\sigma,a,\alpha
,\delta)$.
Then
\[
 \inf \bigl\{ w(y,N) : y \in\mathbb{R}^d, \zeta(y,N)
\geq0 \bigr\} > 0.
\]
Since $\sup_{\mathbb{R}^d} \zeta(\cdot,N) \leq C$ and $w\geq0$ in
$\mathbb{R}^d$,
we obtain
\[
 \zeta(\cdot,N) \leq C w(\cdot,N) \qquad\mbox{in } \mathbb{R}^d.
\]
According to \eqref{e.zetasub} and Lemma~\ref{l.monotonicity},
\[
 \zeta(x,n) \leq C w(x,n) \qquad\mbox{for every } x\in
\mathbb{R}^d, n\geq N.
\]
The lower bound in \eqref{e.theorem1-LB} now follows, since
\begin{eqnarray*}
 &&\inf_{r>0} \liminf_{t\to\infty}
\inf_{|x| \leq r\sqrt{t} } \frac
{\zeta
(x,t)}{\Phi(x,t)} \\
&&\qquad \geq\inf_{r>0}
\liminf_{t\to\infty} \inf_{|x|
\leq r\sqrt{t} } \frac{\Phi_\theta(x,t)}{\Phi(x,t)}
- s\sup_{r>0} \limsup_{t\to\infty} \sup
_{|x| \leq r\sqrt{t} } \frac{\Psi
(x,t)}{\Phi
(x,t)}
\\
& &\qquad\geq\liminf_{t\to\infty} \exp \biggl( \frac{1}\theta
t^{-\theta} \biggr)
\\
&&\qquad\quad{} - s\sup_{r>0} \limsup_{t\to\infty}
t^{\alpha-\beta} \Bigl( \sup_{|y| \leq r } \Psi ( y,1 ) \Bigr) \Bigl(
\inf_{|y| \leq
r } \Phi ( y,1 ) \Bigr)^{-1}
\\
&&\qquad = 1.
\end{eqnarray*}
%
Note that we used that $\beta> \alpha$, from \eqref{e.beta}.
\end{pf*}

\begin{pf*}{Proof of \eqref{e.theorem1-UB} \normalfont{(upper bound)}}
The proof is similar to (and even somewhat easier than) that of the
lower bound. Instead of \eqref{e.zeta}, we use the function
%
\begin{equation}
\label{e.xi} \xi(x,t) := \Phi_{-\theta}(x,t) + \Psi(x,t),
\end{equation}
where $\Psi$ and $\Phi_{-\theta}$ are as in Lemmas \ref{l.Psi}
and \ref
{l.Phitheta2}, respectively. The goal is to show that $\xi$ is a
supersolution of the finite difference equation after a large time.
Then we apply Lemma~\ref{l.monotonicity} to conclude, as above. The
choices of the parameters $\theta$ and
$\beta$ as well as the convention for the constants $C$ and $c$ are the
same as in the proof of the lower bound.

\emph{Step} 1. We show that there exists $T>1$ such that $\xi$ satisfies,
for every $x\in\mathbb{R}^d$ and $t\geq T$,
%
\begin{equation}
\label{e.xisuper} \xi(x,t+1) - \xi(x,t) + F \bigl[ \xi(\cdot,t) \bigr] (x) \geq0.
\end{equation}
Using \eqref{e.poshomo} and \eqref{e.sublinear}, we find that, for
every $t>0$,
\[
 F \bigl[ \xi(\cdot,t) \bigr](x) \geq F \bigl[ \Phi_{-\theta}
(\cdot,t) \bigr](x) + F \bigl[ \Psi(\cdot,t) \bigr](x).
\]
By \eqref{e.Psieq} and \eqref{e.Phitheta2},
\begin{eqnarray*}
&& \xi(x,t+1) - \xi(x,t) + F \bigl[ \xi(\cdot,t) \bigr] (x)\nonumber
\\
&&\qquad \geq t^{-1-\alpha-\theta} \Phi \biggl( \frac{x}{\sqrt t},1 \biggr) -
Ct^{-1-\alpha-\sigma/2} \exp \biggl( -\frac{a|x|^2}{2t} \biggr)
\\
&&\qquad\quad{}
+ c t^{-1-\beta} \Psi \biggl( \frac{x}{\sqrt t},1 \biggr) \cdot
\cases{ - C, &\quad $\mbox{if } |x| \leq C \sqrt{t}$,
\vspace*{2pt}\cr
\displaystyle\frac{|x|} {\sqrt t}, & \quad$\mbox{if }
|x| \geq C\sqrt{t}$.}\nonumber
\end{eqnarray*}
By the choice of parameters, that is,  \eqref{e.theta} and \eqref
{e.beta}, we have for sufficiently large $t$ that \eqref
{e.slowdecaywins} holds for every $|y| \leq C$ and \eqref
{e.fattailswin} holds for every $|y| \geq C$. Together these yield the claim.

\emph{Step} 2. We complete the proof of the upper bound.
Select $T$, $s$, $r$ as in step 1. By the definition of $\Psi$ and
Assumption~\ref{ass.w}, we have
\[
 w(x,N) \leq C \exp \biggl( - \beta\frac{|x|}{\sqrt N} \biggr) \leq
C \Psi (x,N) \leq C \xi(x,N).
\]
[As $T$ depends only on $(d,R,\sigma,a,\alpha)$, the
constant $C$, which depends on $T$, satisfies our convention for dependence
on parameters.]
By \eqref{e.xisuper} and Lemma~\ref{l.monotonicity},
\[
 w(x,n) \leq C \xi(x,n)\qquad \mbox{for every } x\in\mathbb{R}^d,
n\geq N.
\]
We thus conclude the proof of the upper bound by observing that
%
\begin{eqnarray*}
 &&\sup_{r>0} \limsup_{t\to\infty}
\sup_{|x| \leq r\sqrt{t} } \frac
{\xi
(x,t)}{\Phi(x,t)}\\
&&\qquad \leq\sup_{r>0}
\limsup_{t\to\infty} \sup_{|x|
\leq r\sqrt{t} } \frac{\Phi_\theta(x,t)}{\Phi(x,t)}
+ \sup_{r>0} \limsup_{t\to\infty} \sup
_{|x| \leq r\sqrt{t} } \frac{\Psi(x,t)}{\Phi
(x,t)}
\\
& &\qquad\leq\limsup_{t\to\infty} \exp \biggl( -\frac{1}\theta
t^{-\theta} \biggr)
\\
& &\qquad\quad{}+ \sup_{r>0} \limsup_{t\to\infty}
t^{\alpha-\beta} \Bigl( \sup_{|y| \leq r } \Psi ( y,1 ) \Bigr) \Bigl(
\inf_{|y|
\leq r } \Phi ( y,1 ) \Bigr)^{-1}
\\
&&\qquad = 1.
\end{eqnarray*}
We note once again that we used that $\beta> \alpha$, by \eqref{e.beta}.
\end{pf*}

The proof of Theorem~\ref{t.alpha} is now complete.

\section{Existence of self-similar profiles}
\label{s.asscheck}

In this section, we show that Assumptions \ref{ass.Phi} and \ref{ass.w}
hold for a wide class of examples and we indicate methods for computing
(or at least estimating) $\alpha$.

\subsection{A nonlinear principal eigenvalue problem}

In our search for $\Phi$, we may use the self-similarity relation to
reformulate the parabolic equation as an elliptic one. By
differentiating \eqref{e.selfsim}, we have
\[
 \partial_t \Phi(x,t) = -\tfrac{1}2
t^{-1} x\cdot D\Phi(x,t) - \alpha t^{-1} \Phi(x,t)
\]
and substituting this into \eqref{e.Phieq} yields
\[
 F\bigl(D^2\Phi(x,t) \bigr) -\tfrac{1}2
t^{-1} x \cdot D\Phi(x,t) = \alpha t^{-1} \Phi(x,t) \qquad\mbox{in }
\mathbb{R}^d\times(0,\infty).
\]
Using \eqref{e.poshomo} and \eqref{e.selfsim} again to change to the
variable $y = x /\sqrt{t}$, we may eliminate the time variable. We get
%
\begin{equation}
\label{e.Phieqellip} F\bigl(D^2\Phi(y,1)\bigr) - \tfrac{1}2 y \cdot
D\Phi(y,1) = \alpha\Phi(y,1)\qquad \mbox{in } \mathbb{R}^d.
\end{equation}
This is a principal eigenvalue problem: the unknowns $\alpha$ and
$\Phi
(\cdot,1)$ are the principal eigenpair. If we can solve it, then we may
recover the full function $\Phi$ via the self-similarity relation.
While the domain $\mathbb{R}^d$ is unbounded, the drift term makes the problem
well-posed in the uniformly elliptic setting (see \cite{AT}). We
investigate this in more detail in the next subsection.

\subsection{Uniformly elliptic martingales: Proof of
Corollary \texorpdfstring{\protect\ref{cor-elliptic}}{1.1}}
\label{sec-4.2}
In this subsection, we verify that Assumptions \ref{ass.Phi} and \ref
{ass.w} hold for both $F^-$ and $F^+$ in the uniformly elliptic case,
that is, under the additional hypothesis
%
\begin{equation}
\label{e.UEass} \mathcal P \subseteq\mathcal E_{\lambda,R}\bigl(
\mathbb{R}^d\bigr) \qquad\mbox{for some $\lambda> 0, R\geq\sqrt{2d}$.}
\end{equation}
We use \cite{AT}, Theorem~1.1 and the Evans--Krylov theorem \cite{CC},
Theorem~6.6, to show that Assumption~\ref{ass.Phi} holds, while we
verify Assumption~\ref{ass.w}(i) by a pathwise construction.

With $F=F^+$ or $F=F^-$, the results of \cite{AT} imply the existence
of $\alpha> 0$ and $\Phi\in C(\mathbb{R}^d\times(0,\infty))$
satisfying \eqref
{e.Phieq} in the weak viscosity sense as well as \eqref{e.positive}
and \eqref{e.selfsim}. It is also proved that $\Phi$ is unique,
provided we impose the normalization $\Phi(0,1) = 1$, and that the
function $\Phi(\cdot,1)$ satisfies \eqref{e.Phieqellip} and, for some
constants $K_0(d,\lambda)>1$ and $a(d,\lambda)>0$,
%
\begin{equation}
\label{e.wkdecay} \bigl\llvert \Phi(x,1) \bigr\rrvert \leq K_0 \exp
\bigl( -2a|x|^2 \bigr).
\end{equation}
To check Assumption~\ref{ass.Phi}, we have left to show
that the Evans--Krylov theorem and \eqref{e.wkdecay} imply
the stronger bound \eqref{e.decay}. This is handled in the following lemma.

\begin{lemma}
Let $F$, $\lambda$, $\alpha$, $\Phi$, $K_0$ and $a>0$ be as above. Then
there exist $\sigma(d,\lambda)\in(0,1]$ and $C(d,\lambda,\alpha
,K_0,a)>0$ such that, for every $x\in\mathbb{R}^d$,
\[
 \bigl\llVert \Phi(\cdot,1) \bigr\rrVert _{C^{2,\sigma}(B_1(x))} \leq C
\exp \bigl( -a|x|^2 \bigr).
\]
\end{lemma}
\begin{pf}
Throughout the argument, $C$ denotes a positive constant depending only
on $(d,\lambda,\alpha,K_0,a)$ which may vary in each occurrence. As
mentioned above, the function $\varphi(x):= \Phi(x,1)$ satisfies the equation
%
\begin{equation}
\label{e.eig} F\bigl(D^2 \varphi\bigr) - \tfrac{1}2 x\cdot D
\varphi= \alpha\varphi\qquad\mbox {in } \mathbb{R}^d.
\end{equation}
For the moment, we must interpret \eqref{e.eig} in the weak viscosity
sense (as defined in~\cite{CC}), although we will see shortly that
$\varphi$ is $C^2$ and, therefore, \eqref{e.eig} can be understood in
the classical sense.

Equation \eqref{e.eig} possesses a local length scale arising from the
competition between the gradient term and the diffusive term. Since the
gradient term is stronger for larger $|x|$, in order to apply local
elliptic regularity estimates to $\varphi$ near $x\in\mathbb{R}^d$,
it is
natural to rescale the equation in some way which depends on $|x|$. We
perform the rescaling by introducing the variable $y=x/r$ and denoting
$\varphi_r(y):= \varphi(x)$, where $0< r \leq1$. In terms
of $\varphi
_r$, equation \eqref{e.eig} takes the form
%
\begin{equation}
\label{e.eigr} F\bigl(D^2 \varphi_r\bigr) -
\tfrac{1}2r^2 y\cdot D\varphi= \alpha r^2
\varphi_r \qquad\mbox{in } \mathbb{R}^d.
\end{equation}
Notice that the first-order coefficient is uniformly bounded and smooth
for every $|y| \lesssim r^{-2}$. The interior gradient H\"older
estimate for uniformly elliptic equations therefore implies that
$\varphi_r \in C^{1,\sigma}(B_4(y_0))$ for some $\sigma(d,\lambda
)\in
(0,1]$ and, for every $0< r\leq1$ and $y_0\in\mathbb{R}^d$ with
$|y_0| \leq r^{-2}$,
%
\begin{equation}
\label{e.varphc1a} \llVert \varphi_r \rrVert _{C^{1,\sigma}(B_4(y_0))} \leq C
\bigl( 1+\alpha r^2 \bigr) \llVert \varphi_r \rrVert
_{L^\infty(B_8(y_0))}.
\end{equation}
The standard reference for this estimate for equations with no gradient
dependence is \cite{CC}, Theorem~8.3, and the argument there can be
adapted in
a straightforward manner to handle equations with gradient dependence
and, in particular, \eqref{e.eigr}. Alternatively, we refer to \cite{T}
for a statement with hypotheses covering our case.

We may now re-express \eqref{e.eigr} as
\[
F\bigl(D^2 \varphi_r\bigr) = f \qquad\mbox{in }
\mathbb{R}^d,
\]
where, in view of \eqref{e.varphc1a}, the function $f(y):= \frac{1}2r^2
y\cdot D\varphi(y) + \alpha r^2 \varphi_r(y)$ satisfies, for each $0<
r\leq1$ and $|y_0| \leq r^{-2}$,
\[
 \llVert f \rrVert _{C^{\sigma}(B_4(y_0))} \leq C \bigl( 1+\alpha
r^2 \bigr) \llVert \varphi_r \rrVert _{L^\infty(B_8(y_0))}.
\]
As $F$ is either convex or concave, the Evans--Krylov theorem
(cf. \cite{CC}, Theorem~6.6) yields [after redefining $\sigma
(d,\lambda
)$ to be smaller, if necessary] that $ \varphi_r \in C^{2,\alpha
}(B_2(y_0))$ and gives the estimate
%
\begin{eqnarray}
\label{e.c2sigr} \llVert \varphi_r \rrVert _{C^{2,\sigma}(B_2(y_0))} & \leq& C
\bigl( \| \varphi_r\|_{L^\infty(B_4(y_0))} + \llVert f \rrVert
_{C^{\sigma
}(B_4(y_0))} \bigr)
\nonumber
\\[-8pt]
\\[-8pt]
\nonumber
& \leq& C \bigl( 1+\alpha r^2 \bigr) \llVert \varphi_r
\rrVert _{L^\infty(B_8(y_0))}.
\end{eqnarray}
We now reverse the scaling to express \eqref{e.c2sigr} in terms
of $\varphi$. For a fixed $x_0\in\mathbb{R}^d$, set $r:= \frac{1}4(1+|x_0|)^{-1}$
and $y_0:= x_0 / r$. Note that $|y_0| < r^{-2}$ and that we have
\begin{eqnarray*}
 \llVert \varphi\rrVert _{L^\infty(B_{2r}(x_0))}&=& \llVert
\varphi_r \rrVert _{L^\infty(B_{2}(y_0))}\quad \mbox{and}\\
 \llVert \varphi \rrVert
_{C^{2,\sigma}(B_{2r}(x_0))} &\leq& r^{-2-\sigma} \llVert \varphi _r \rrVert
_{C^{2,\sigma}(B_{2}(y_0))}.
\end{eqnarray*}
We therefore obtain from \eqref{e.c2sigr} that
\[
 \llVert \varphi\rrVert _{C^{2,\sigma}(B_{2r}(x_0))} \leq Cr^{-2-\sigma
}
\bigl( 1+\alpha r^2 \bigr) \llVert \varphi\rrVert _{L^\infty
(B_{2r}(x_0))}.
\]
According to \cite{AT}, the exponent $\alpha$ depends only on
$(d,\lambda)$. Applying \eqref{e.wkdecay} and using the fact that
$B_{8r}(x_0) \subseteq\mathbb{R}^d\setminus B_{3|x_0|/4}$ for $|x_0|
> 4$, we obtain
\[
 \llVert \varphi\rrVert _{C^{2,\sigma}(B_{2r}(x_0))} \leq C
\bigl(1+|x_0| \bigr) ^{2+\sigma} \exp \bigl( -2a \cdot\tfrac{9}{16}
|x_0|^2 \bigr) \leq C \exp \bigl( -\tfrac{17}{16}
a |x_0|^2 \bigr). 
\]
The previous estimate implies
\[
 \llVert \varphi\rrVert _{C^{2,\sigma}(B_{1}(x_0))} \leq C \exp \bigl( -a
|x_0|^2 \bigr). 
\]
\upqed\end{pf}

Checking Assumption~\ref{ass.w} in the uniformly elliptic case is
straightforward. As we previously mentioned, the upper bound \eqref
{e.asstails} follows from Azuma's inequality. To check \eqref
{e.assbulk}, it is enough to consider a simple random walk. That is, we
take $\{\Delta_i\}_{i\geq0}$ as a sequence of i.i.d. random variables
so that
$\P(\Delta_0=\pm\sqrt{2d} e_i)=1/2d$ for $e_i$ the standard unit
vectors in $\R^d$. Let $y_x$ denote an element of $\sqrt{2d}\Z^d$ with
minimal norm $|x-y_x|$. Then, with $|y_x|_1$ denoting the $\ell^1$ norm
of $y_x/\sqrt{2d}$, we have, for every $n>|y_x|_1$,
\[
\P\bigl(|M_n-y_x|\leq\sqrt{2}d\bigr)\geq 
 \biggl(
\frac{1}{2d} \biggr)^n.
\]
This immediately implies \eqref{e.assbulk}. This completes the proof of
Corollary~\ref{cor-elliptic}.


%

\subsection{Estimating the exponent $\alpha$:
The proof of \texorpdfstring{\protect\eqref{e.alphabounds}}{(1.10)} and 
\texorpdfstring{\protect\eqref{e.alphabounds2}}{(1.11)}}

As we have seen, finding the exponent $\alpha$ and self-similar profile
$\Phi$ is equivalent to solving a nonlinear eigenvalue problem. This
is of course difficult in general, both analytically and
computationally. Even for particular examples like $\mathcal P =
\mathcal E_{\lambda,R}$, in which case $F^-$ and $F^+$ are the minimal
and maximal Pucci operators, respectively, we do not believe it is
possible to give a closed form expression for $\alpha$ or $\Phi$
(although for rotationally invariant operators like these, the problem
can be reduced to an ODE in the radial variable, which greatly reduces
its complexity).

Fortunately, it is more tractable to estimate $\alpha$. This can be
done by
exhibiting explicit $\beta$ for which there exist subsolutions and
supersolutions of the equation
%
\begin{equation}
\label{e.eigbeta} F\bigl(D^2g\bigr) - \tfrac{1}2 x \cdot Dg =
\beta g \qquad\mbox{in } \mathbb{R}^d.
\end{equation}
Let $X$ denote the space
\[
 X:= \bigl\{ g \in C^{2}\bigl(\mathbb{R}^d
\bigr) : \mbox{there exists } a>0 \mbox{ such that } 0\leq g(x) \leq\exp
\bigl(-a|x|^2 \bigr) \bigr\}
\]
and set $X_+ := X \cap\{ g > 0\}$. The following formulas for $\alpha$
were proved in \cite{AT}:
%
\begin{equation}
\label{e.low} \alpha= \sup \bigl\{ \beta> 0 : \exists g\in X_+ \mbox
{ satisfying } F\bigl(D^2g\bigr) - \tfrac{1}2 x \cdot Dg \geq
\beta g \mbox { in } \mathbb{R}^d \bigr\}
\end{equation}
and
%
\begin{equation}\qquad
\label{e.high} \alpha= \inf \bigl\{ \beta> 0 : \exists g\in X_+ \mbox
{ satisfying } F\bigl(D^2g\bigr) - \tfrac{1}2 x \cdot Dg \leq
\beta g \mbox { in } \mathbb{R}^d \bigr\}.
\end{equation}
This allows us to bound $\alpha$ from below (resp., above) by
exhibiting a supersolution (resp., subsolution) of \eqref
{e.eigbeta} with an explicit $\beta$.

In the case that $\mathcal P = \mathcal E_{\lambda,R}$, test functions
were found in \cite{AT} that give the bounds
%
\begin{equation}
\label{e.ATbounds} \frac{d \lambda}{2} \leq\alpha\bigl(F^-\bigr) \leq
\frac{(d-1)\lambda}{2} +\frac{1}2 \leq\frac{(d-1)}{2\lambda} +\frac{1}2
\leq\alpha\bigl(F^+\bigr) \leq\frac
{d}{2\lambda},
\end{equation}
where, for each of these inequalities, equality holds only if $\lambda
=1$. in particular, $\alpha(F^-) < \frac{d}2 < \alpha(F^+)$ if
$\lambda
< 1$. A more general fact along these lines was shown in \cite{AT}, Example~3.12, namely that if $\mathcal P \subseteq\mathcal E_{\lambda
,R}$, then $\alpha(F^-) < \frac{d}2 < \alpha(F^+)$ unless $\mathcal P$
is a singleton set, that is, unless the controller has no actual
control and the martingale is just a simple random walk.

In the next lemma, we use \eqref{e.low} and \eqref{e.high} to prove the
bounds \eqref{e.alphabounds} as promised in the \hyperref[sec1]{Introduction}. To aid
our computation, we remark that for $\mathcal P = \mathcal E_{\lambda
,R}$, the operator $F^-$ can be expressed for each $M\in\mathbb{S}^d$ as
%
\begin{eqnarray}
\label{e.pucciform} F^-(M) & = &- \lambda\cdot ( \mbox{sum of the negative eigenvalues
of } M )
\nonumber
\\[-8pt]
\\[-8pt]
\nonumber
&&{} - ( \mbox{sum of the positive eigenvalues of } M ).
\end{eqnarray}

\begin{lemma}
\label{l.alphcalc}
In the case that $\mathcal P = \mathcal E_{\lambda,R}$ and $F=F^-$, for
every $\delta>0$, there exist $C(d,\delta)>1$ and $c(d,\delta)>0$
such that
\[
 c\lambda^{1/4+\delta} \leq\alpha\leq C \lambda^{1/4-\delta}.
\]
\end{lemma}

\begin{pf*}{Proof of Lemma~\ref{l.alphcalc} \normalfont{(upper bound)}}
Fix $p\in(0,1/2)$ and parameters $a,b>0$ to be selected below, and
consider the test function
%
\begin{equation}
\label{e.varph} \varphi(x):= \exp \bigl( -\tfrac{1}2a |x|^2 -
b \bigl( \lambda+|x|^2 \bigr)^{p/2} \bigr).
\end{equation}
According to \eqref{e.high}, it suffices to show, for appropriate
choices of $a$ and $b$, that $\varphi$ satisfies
%
\begin{equation}
\label{e.testfunn} F^-\bigl(D^2\varphi(x) \bigr) - \tfrac{1}2 x
\cdot D\varphi(x) \leq C \lambda ^{p/2}\varphi\qquad \mbox{in }
\mathbb{R}^d.
\end{equation}
Here and throughout the rest of the argument, $C$ denotes a positive
constant depending only on $(p,d)$ which may vary in each occurrence.

\emph{Step} 1. We compute the first two derivatives of $\varphi$ and
estimate $F^-(D^2\varphi(x))$ from above. We have
\[
D\varphi(x) = -\varphi(x) \bigl( a + bp \bigl( \lambda+|x|^2
\bigr)^{p/2-1} \bigr) x
\]
and
%
\begin{eqnarray}
\label{e.varphess} D^2\varphi(x) & = &\varphi(x) \bigl( a + bp \bigl(
\lambda +|x|^2 \bigr)^{p/2-1} \bigr)^2 x \otimes x
\nonumber\\
&&{} + \varphi(x) \bigl( bp(2-p) \bigl( \lambda+|x|^2
\bigr)^{p/2-2} \bigr) x \otimes x
\\
&&{} -\varphi(x) \bigl( a + bp \bigl( \lambda+|x|^2
\bigr)^{p/2-1} \bigr) I.
\nonumber
\end{eqnarray}
Discarding some of the terms coming from expanding the square on the
first term on the right in the expression for $D^2\varphi(x)$ above, we
find that
\begin{eqnarray*}
D^2\varphi(x) \geq M(x) & :=&\varphi(x) \bigl( a^2 +
bp(2-p) \bigl( \lambda+|x|^2 \bigr)^{p/2-2} \bigr) x \otimes x
\\
&&{} -\varphi(x) \bigl( a + bp \bigl( \lambda+|x|^2
\bigr)^{p/2-1} \bigr) I.
\end{eqnarray*}
Observe that
%
\begin{eqnarray*}
&&\mbox{eigenvalues of $M(x)$} 
\\
&&\qquad= \varphi(x) \cdot\cases{ E(x), &\quad  $
\mbox{with multiplicity } 1$,
\vspace*{2pt}\cr
-a - bp \bigl( \lambda+|x|^2
\bigr)^{p/2-1}, &\quad  $\mbox{with multiplicity } d-1$,} 
\end{eqnarray*}
where
\begin{eqnarray*}
 E(x) & =& a^2|x|^2 - a + bp(2-p) \bigl(
\lambda+|x|^2 \bigr)^{p/2-2}|x|^2 - bp \bigl(
\lambda+|x|^2 \bigr)^{p/2-1}
\\
& = &a^2|x|^2 - a + bp(1-p) \bigl( \lambda+|x|^2
\bigr)^{p/2-1} - bp(2-p) \lambda \bigl( \lambda+|x|^2
\bigr)^{p/2-2}.
\end{eqnarray*}
Using \eqref{e.pucciform}, we get
%
\begin{eqnarray}
\label{e.evalvarph}&& F^-\bigl(D^2 \varphi(x) \bigr)\nonumber \\
&&\qquad\leq F^-\bigl(M(x)
\bigr)
\\
&&\qquad= \varphi(x)\cdot\cases{ \lambda(d-1) \bigl( a + bp \bigl( \lambda+|x|^2
\bigr)^{p/2-1} \bigr) -\lambda E(x), &\quad  $\mbox{if } E(x) \leq0$,
\cr
\lambda(d-1) \bigl( a + bp \bigl( \lambda+|x|^2 \bigr)^{p/2-1}
\bigr) - E(x), &\quad $\mbox{if } E(x) > 0$.}\nonumber
\end{eqnarray}

\emph{Step} 2. We study the set where $E(x)$ is positive. The claim is
that, for $a\geq1$, there exists $C>1$ such that $b \geq C$ implies
%
\begin{eqnarray}
\label{e.Epos} E(x) \geq\frac{1}2 a|x|^2 +
\frac{1}2bp|x|^2 \biggl( 1+\frac
{|x|^2}{\lambda
}
\biggr)^{p/2-1} \geq0
\nonumber
\\[-8pt]
\\[-8pt]
\eqntext{\mbox{for every } |x| > C \sqrt{\lambda}.}
\end{eqnarray}
This follows from the following three facts, each of which is easy to check
\begin{eqnarray*}
&& |x|^2 \geq\frac{2}{a} \quad\Longrightarrow\quad
\frac{1}2 a^2|x|^2 \geq a,
\\
 &&|x|^2 \leq\frac{2}{a} \quad\mbox{and}\quad b \geq C
\quad\Longrightarrow\quad\frac{1}4 bp(1-p) \bigl(\lambda+|x|^2
\bigr)^{p/2-1} \geq a,
\\
&& |x|^2 \geq C \lambda\quad\Longrightarrow\quad \frac{1}4
bp(1-p) \bigl(\lambda+|x|^2 \bigr)^{p/2-1} \geq bp(2-p)
\lambda \bigl(\lambda+|x|^2 \bigr)^{p/2-2}.
\end{eqnarray*}

\emph{Step} 3. We check \eqref{e.testfunn} for $|x| \geq C\sqrt
{\lambda
}$. Note that the estimate \eqref{e.Epos} says precisely that
\[
 E(x) \varphi(x) \geq-\tfrac{1}2 x\cdot D\varphi(x)\qquad
\mbox{for every } |x| \geq C\sqrt\lambda.
\]
This therefore allows us to absorb the gradient term on the left-hand
side of \eqref{e.testfunn}. Using \eqref{e.evalvarph} and taking now
$a:=1$, we find that
\begin{eqnarray*}
 F^-\bigl(D^2\varphi(x)\bigr) -\tfrac{1}2 x
\cdot D\varphi(x) & \leq&\varphi (x)\lambda(d-1) \bigl( a + bp \bigl(
\lambda+|x|^2 \bigr)^{p/2-1} \bigr)
\\
& \leq&\varphi(x)\lambda(d-1) \bigl( a + bp \lambda^{p/2-1} \bigr)
\\
&\leq& C \lambda^{p/2}\varphi(x).
\end{eqnarray*}

\emph{Step} 4. We check \eqref{e.testfunn} in the set $|x| \leq
C\sqrt
{\lambda}$. Here, we get the estimate \eqref{e.testfunn} differently,
since the gradient term does not hurt us:
\[
 -\tfrac{1}2 x\cdot D\varphi(x) \leq\varphi(x)-\varphi(x)
\bigl( a + bp \bigl( \lambda+|x|^2 \bigr)^{p/2-1} \bigr)
|x|^2 \leq C \lambda ^{p/2} \varphi(x).
\]
A similar estimate yields, for $|x| \leq C\sqrt{\lambda}$,
\[
 -\lambda E(x)\varphi(x) \leq\lambda \bigl( a + bp \bigl(
\lambda+|x|^2 \bigr)^{p/2-1} \bigr) \varphi(x) \leq C
\lambda^{p/2} \varphi(x).
\]
Returning to \eqref{e.evalvarph}, we obtain that, for $|x| \leq C\sqrt
{\lambda}$,
\[
 F^-(D^2\varphi(x) - \tfrac{1}2 x \cdot D
\varphi(x) \leq C \lambda^{p/2} \varphi(x).
\]
This completes the proof of \eqref{e.testfunn}.
\end{pf*}

\begin{pf*}{{Proof of Lemma~\ref{l.alphcalc} \normalfont{(lower bound)}}}
We use the same test function $\varphi$ from the previous argument,
except here we take $p\in(1/2,1]$. The goal is to show that, for
appropriate choices of the parameters $a$ and $b$ (here we take them to
be very small, depending on $p$), we have the reverse of \eqref{e.testfunn}:
%
\begin{equation}
\label{e.testfunn2} F^-\bigl(D^2\varphi(x) \bigr) - \tfrac{1}2 x
\cdot D\varphi(x) \geq c \lambda ^{p/2}\varphi\qquad\mbox{in }
\mathbb{R}^d.
\end{equation}
The analysis and computations involved are quite similar to
those of the previous argument, and so we leave the details to the reader.
\end{pf*}

\section*{Acknowledgements}
We thank Yuval Peres and Charlie Smart for helpful discussions, as well
as for providing us with their analysis of
the motivating example in Corollary~\ref{cor-fixednorm}.
We thank the referees and Associate Editor for useful comments that helped
in improving the presentation of the results.

%
%





\printaddresses
\end{document}